\documentclass[article, 11pt]{amsart}
\usepackage[top = 1.5in, bottom = 1.5in, left=1.2in, right=1.2in]{geometry}
\usepackage{setspace}

\usepackage{tikz}
\usetikzlibrary{decorations.pathmorphing, arrows.meta}

\usepackage[utf8]{inputenc}
\usepackage[T1]{fontenc}

\usepackage{amssymb,amsmath,amsthm,graphicx,xcolor,mathtools,mathrsfs}
\usepackage{comment}

\usepackage{hyperref}
\hypersetup{colorlinks, linkcolor={red!50!black}, citecolor={green!50!black}, urlcolor={blue!50!black}}

\usepackage{cleveref}
\usepackage{nicefrac}

\usepackage{csquotes}

\usepackage{cleveref}
\theoremstyle{plain}
\newtheorem{theorem}{Theorem}[section]
\crefname{theorem}{Theorem}{Theorems}

\crefname{proposition}{Proposition}{Propositions}

\newtheorem{corollary}[theorem]{Corollary}
\crefname{corollary}{Corollary}{Corollaries}

\newtheorem{lemma}[theorem]{Lemma}
\crefname{lemma}{Lemma}{Lemmas}

\crefname{conjecture}{Conjecture}{Conjectures}

\crefname{problem}{Problem}{Problem}

\crefname{claim}{Claim}{Claims}

\crefname{observation}{Observation}{Observations}

\crefname{setup}{Setup}{Setups}

\crefname{myth}{Myth}{Myths}

\crefname{fact}{Fact}{Facts}

\crefname{algorithm}{Algorithm}{Algorithms}

\crefname{remark}{Remark}{Remarks}

\crefname{example}{Example}{Examples}

\theoremstyle{definition}

\crefname{definition}{Definition}{Definitions}

\crefname{construction}{Construction}{Constructions}

\crefname{question}{Question}{Questions}

\numberwithin{equation}{section}

\usepackage[shortlabels]{enumitem}
\setlist[enumerate,1]{label={\upshape (\roman*)}}

\newcommand{\eps}{\varepsilon}
\newcommand{\calP}{\mathcal{P}}
\newcommand{\Oh}{\mathrm{O}}

\newcommand{\ssp}{\mathrm{ssp}}

\allowdisplaybreaks

\author[C. G. Fernandes]{Cristina G. Fernandes}

\author[C. Hoppen]{Carlos Hoppen}
\address[C. Hoppen]{Instituto de Matemática e Estatística, Universidade Federal do Rio Grande do Sul,
  Porto Alegre, Brazil.}
\email{choppen@ufrgs.br}

\author[G. Kontogeorgiou]{George Kontogeorgiou}
\address[G. Kontogeorgiou]{Centro de Modelamiento Matemático,
  Universidad de Chile, Santiago, Chile.}
\email{gkontogeorgiou@dim.uchile.cl}

\author[G. O. Mota]{Guilherme Oliveira Mota}

\address[C. G. Fernandes and G. O. Mota]{Instituto de Matemática, Estatística e Ciência da Computação,
  Universidade de São Paulo, Rua do Matão 1010, 05508-090 São Paulo, Brazil.}
\email{\{\,cris\,|\,mota\,\}@ime.usp.br}

\author[D.~Peng]{Danni Peng}
\address[D. Peng]{Instituto Nacional de Matemática
    Pura e Aplicada, Rio de Janeiro, Brazil.}
\email{danni.peng@impa.br}

\thanks{We would like to thank Maya Stein for inviting us to the $2^{nd}$ Graph Theory in the Andes Workshop, where this work was undertaken.
This study was partly supported by a joint project FAPESP and ANID
(2019/13364-7) and by CAPES (Finance Code 001).
C.G.\ Fernandes was supported by CNPq (312511/2025-6, 404315/2023-2) and FAPESP (2023/03167-5).
C.\ Hoppen was supported by CNPq (315132/2021-3 and 408180/2023-4).
G.\ Kontogeorgiou was supported by ANID (Grant CMM Basal FB210005)
and ANID-FONDECYT Postdoctorado Grant No.~3250479.
G.O.\ Mota was supported by CNPq (315916/2023-0 and 406248/2021-4) and FAPESP (2023/03167-5 and 2024/13859-4).
D. \ Peng was supported by CNPq (141537/2023-0).
}

\title[Separating paths for cubic graphs and complete bipartite graphs]
{Separating path systems for cubic graphs \\ and for complete bipartite graphs}

\begin{document}
\onehalfspace

\begin{abstract}
A strongly separating path system in a graph $G$ is a collection $\calP$ of paths in $G$ such that, for every two edges $e$ and $f$ of $G$, there is a path in $\calP$ with $e$ but not~$f$, and vice-versa. The minimum size of such a system is the so called strong separation number of $G$.
We prove that the strong separation number of every $2$-degenerate graph on $n$ vertices is at most~$n$. Using this, we also provide upper bounds for the strong separation number of subcubic graphs, planar graphs, and planar bipartite graphs. On the other hand, we prove that the strong separation number of a complete bipartite graph $K_{a,b}$ is at least~$b$ if $a<b/2$ and at least $(\sqrt{6(b/2)+4}-2)a$ if $b/2\leq a\leq b$, and we provide a construction that attains the former bound.
\end{abstract}

\maketitle

\section{Introduction}
\label{sec:intro}

Given a collection $\calP$ of paths in a graph $G$, we say that two edges 
$e,f$ in $G$ are \emph{separated} by~$\calP$ if there are two paths $P_e$ and~$P_f$ 
in $\calP$ such that $P_e$ contains $e$ but avoids~$f$, and~$P_f$ contains
$f$ but avoids $e$. We say that $\calP$ is a \emph{strongly separating path
system} of~$G$ if~$\calP$ separates every pair of edges in~$G$.
The minimum size of a strongly separating path system of $G$ is denoted by $\ssp(G)$.

Balogh, Csaba, Martin, and Pluhár~\cite{BCMP2016} conjectured that $\ssp(G)=\Oh(n)$ 
for every $n$-vertex graph~$G$, and succeeded in establishing an upper bound of $\Oh(n \log n)$ 
on $\ssp(G)$.  A substantial improvement to $\Oh(n \log^\ast n)$ was made by Letzter~\cite{Letzter2024}.
Finally, Bonamy, Botler, Dross, Naia, and Skokan~\cite{BBDNS2023} proved an upper bound of $19n$, 
settling the conjecture as true.

The bound of $19n$ is not tight and it is possible that $\ssp(G) \leq (1+o(1))\,n$ 
for every connected $n$-vertex graph $G$,\footnote{In \cite{BCMP2016}, the authors claim an example of a graph $G$ with $\ssp(G)\geq (2-\eps)n$, but their proof has a flaw that invalidates this statement: the length of a longest path in $K_{\eps n, (1 - \eps)n}$ is $2 \eps n$ and not $\eps n + 1$.}
where the connectivity condition is necessary, as the graph consisting of $n/4$ disjoint 
copies of $K_4$ requires $5n/4$ paths to be strongly separated. Recently, the first and 
fourth authors together with Sanhueza-Matamala~\cite{FSMM25} proved that
$\ssp(K_n) = (1+o(1))n$ and $\ssp(K_{\frac{n}{2}, \frac{n}{2}})=(\sqrt{5/2} - 1 + o(1))\,n$, and 
in general obtained the asymptotic value of $\ssp(G)$ for $\alpha n$-regular $n$-vertex graphs 
that are \emph{robustly connected}. For complete graphs, this result was further improved by 
the third author and Stein~\cite{kontogeorgiou2024exactupperboundminimum}, who constructed a 
strongly separating path system for~$K_n$ with at most $n+9$ paths. In this work, we continue 
this line of research by establishing new bounds for 2-degenerate graphs (Section \ref{sec2}) 
and for complete bipartite graphs (Section \ref{sec4}).

A graph $G$ is \emph{$2$-degenerate} if every subgraph of~$G$ contains a vertex of degree at 
most~$2$. The class of 2-degenerate graphs includes, for instance, all outerplanar graphs, 
all series-parallel graphs, all strictly subcubic graphs, and all planar graphs of girth 
at least~6. Our first main result is the following.

\begin{theorem}\label{thm:2degenerate}
  For every $2$-degenerate $n$-vertex graph $G$, we have $\ssp(G) \leq n$.
\end{theorem}

Our proof is inductive and yields a separating path system $\mathcal{P}$ with some very convenient properties. Namely, if every connected component of $G$ contains at least three vertices, then $\mathcal{P}$ is such that every edge of $G$ is contained in exactly two of its paths, and every vertex of $G$ is the endpoint of exactly two of its paths. In Section~\ref{sec3}, we apply Theorem~\ref{thm:2degenerate} to obtain upper bounds for $\ssp(G)$ when $G$ is subcubic, planar, or planar bipartite.

We also study separating path systems for unbalanced complete bipartite graphs $K_{a,b}$, $a \leq b$. If $a<b/2$, we show that $\ssp(K_{a,b})=b$. For $a\geq b/2$, we establish a lower bound of $(\sqrt{6(b/a)+4}-2)a$ on $\ssp(K_{a,b})$. This bound is tight at the extremes, that is, when~$a = b/2$ and when $a=b$.

In what follows, we refer to a path just by its sequence of vertices. That is, for a path with vertex set $\{v_1,\ldots,v_\ell\}$ where $v_iv_{i+1}$ is an edge for every $1 \leq i \leq \ell-1$, we write $P=(v_1,\dots,v_\ell)$.

\section{2-Degenerate graphs} \label{sec2}
\label{sec:2degenerate}

In this section, we prove that $\ssp(G) \leq n$ for every 2-degenerate $n$-vertex graph~$G$. 
In fact, we prove a stronger statement which is useful for our other results and gives an 
extra structural property for our strongly separating path systems for $2$-degenerate graphs.


\begin{theorem}\label{thm:2degenerate-strong}
  For every connected $2$-degenerate graph $G$ with at least three vertices, 
  there is a strongly separating path system for $G$ with $n$ paths and
  the following two properties:
  \begin{enumerate}
  \item every edge lies in exactly two paths;\label{item:i}
  \item there are exactly two paths ending at each vertex of~$G$.\label{item:ii}
  \end{enumerate}
\end{theorem}

Before going into the proof, let us argue that this implies that $\ssp(G) \leq n$ for every 2-degenerate $n$-vertex graph~$G$.
Indeed, let $G_1,\ldots,G_k$ be the connected components of~$G$.
For each $G_i$ with at least three vertices, take the strongly separating path system given by Theorem~\ref{thm:2degenerate-strong}.
For a $G_i$ that consists of a single edge, take the strongly separating path system consisting of a single path of length~1 (the edge itself),
and for a $G_i$ that consists of a single vertex, take the empty separating path system.
Each such path system has at most $|V(G_i)|$ paths, and 
the union of these path systems is a strongly separating path system for $G$,
hence $\ssp(G) \leq \sum_{i=1}^k|V(G_i)| = n$.

\begin{proof}[Proof of Theorem~\ref{thm:2degenerate-strong}.]
  The proof is by induction on~$n$. If $n = 3$, then $G$ is either a path of length $2$ or a triangle.
  If $G$ is a path of length~$2$, say $(v_1,v_2,v_3)$, the family $\{(v_1,v_2,v_3),\,(v_1,v_2),\\(v_2,v_3)\}$
  is a strongly separating path system for $G$ satisfying~\ref{item:i} and~\ref{item:ii}.
  If $G$ is a triangle with vertices $\{v_1,v_2,v_3\}$, then our strongly separating path system
  satisfying~\ref{item:i} and~\ref{item:ii} is $\{(v_1,v_2,v_3),\,(v_2,v_3,v_1),\,(v_3,v_1,v_2)\}$.

  For the induction step, consider $n \geq 4$, and assume that the result holds
  for any connected 2-degenerate graph with fewer than~$n$ vertices.
  
  Firstly, suppose that $G$ has a vertex $v$ of degree at most $2$ for which
  $G-v$ is connected. If~$v$ has degree~1, consider a separating path
  system $\calP'$ for $G-v$ satisfying~\ref{item:i} and~\ref{item:ii}, of size~$|\calP'|=n-1$,
  and let $P_u$ be one of the paths in $\calP'$ that ends at the neighbor
  $u$ of~$v$. Let $P$ be the path that extends $P_u$ to $v$ (see Figure~\ref{fig:3cases}(a)).
  The path family $\calP:=(\calP'\setminus\{P_u\})\cup\{P,(u,v)\}$ is a strongly
  separating path system such that every edge lies in exactly two paths and
  there are exactly two paths ending at each vertex of~$G$, that is,
  $\calP$ satisfies \ref{item:i} and~\ref{item:ii}.
  If $v$ has degree $2$, let $u$ and $w$ be its neighbors,
  and let~$\calP'$ be a strongly separating path system for $G-v$ of
  size $|\calP'|=n-1$ satisfying \ref{item:i} and~\ref{item:ii}.
  Let $P_u$ be one of the paths in $\calP'$ that ends in $u$ and let $P_w$ be a
  path in $\calP'$ that ends in $w$ with the property that $P_w \neq P_u$.
  Let~$P_1$ and $P_2$ be the paths that extend $P_u$ and~$P_w$ to $v$
  using the edges $uv$ and $vw$, respectively (see Figure~\ref{fig:3cases}(b)).
  The path family $\calP:=(\calP'\setminus\{P_u,P_w\})\cup\{P_1,P_2,(u,v,w)\}$ has
  the property that every edge lies in exactly two paths and
  there are exactly two paths ending at each vertex of~$G$.
  Note that $P_1$ and~$P_2$ strongly separate $uv$ from~$vw$.
  The edges $uv$ and $vw$ are separated from the remaining edges by $(u,v,w)$.
  Finally, by induction, any edge~$e'$ in $G-v$ lies in a path other than $P_u$
  (and so corresponds to a path in~$\calP$ that avoids $uv$) and lies
  in a path other than $P_w$ (and so corresponds to a path in~$\calP$
  that avoids $vw$). Therefore, $\calP$ is a strongly separating path system 
  for~$G$ satisfying properties~\ref{item:i} and~\ref{item:ii}.

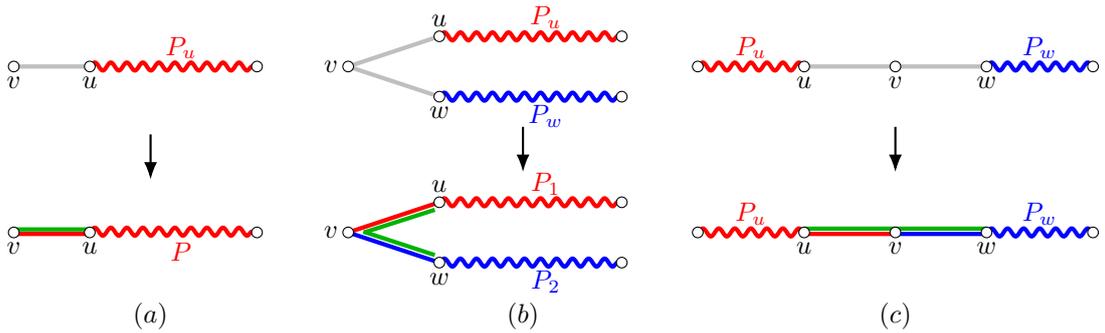
\begin{figure}[htbp]
  \centering
  \begin{tikzpicture}[scale=1, every node/.style={font=\small}]
    \draw[lightgray, ultra thick] (-2,3.2) -- (-1,3.2);
    \draw[decorate, decoration={snake, segment length=6pt, amplitude=1.5pt}, red, ultra thick] (-1,3.2) -- (1.2,3.2);
    \filldraw[fill=white, draw=black] (-2,3.2) circle (2pt) node[below] {$v$};
    \filldraw[fill=white, draw=black] (-1,3.2) circle (2pt) node[below] {$u$};
    \filldraw[fill=white, draw=black] (1.2,3.2) circle (2pt);
    \node[red] at (0.2,3.45) {$P_u$};

    \draw[thick, -{Latex}] (-0.2,2.3) -- (-0.2,1.7);
    \node[black] at (-0.2,-0.1) {$(a)$};

    \draw[green!70!black, ultra thick] (-2,1.04) -- (-1,1.04);
    \draw[red, ultra thick] (-2,0.98) -- (-1,0.98);
    \draw[decorate, decoration={snake, segment length=6pt, amplitude=1.5pt}, red, ultra thick] (-1,1.01) -- (1.2,1.01);

    \filldraw[fill=white, draw=black] (-2,1) circle (2pt) node[below] {$v$};
    \filldraw[fill=white, draw=black] (-1,1) circle (2pt) node[below] {$u$};
    \filldraw[fill=white, draw=black] (1.2,1) circle (2pt);
    \node[red] at (0.2,0.75) {$P$};

    \begin{scope}[xshift=4.8cm,yshift=0mm]
      \draw[lightgray, ultra thick] (-2.4,3.2) -- (-1.2,3.6); 
      \draw[lightgray, ultra thick] (-2.4,3.2) -- (-1.2,2.8); 

      \draw[decorate, decoration={snake, segment length=6pt, amplitude=1.5pt}, red, ultra thick] (-1.2,3.6) -- (1.2,3.6);
      \draw[decorate, decoration={snake, segment length=6pt, amplitude=1.5pt}, blue, ultra thick] (-1.2,2.8) -- (1.2,2.8);

      \filldraw[fill=white, draw=black] (-2.4,3.2) circle (2pt) node[left] {$v$};
      \filldraw[fill=white, draw=black] (-1.2,3.6) circle (2pt) node[above] {$u$};
      \filldraw[fill=white, draw=black] (-1.2,2.8) circle (2pt) node[below] {$w$};
      \filldraw[fill=white, draw=black] (1.2,3.6) circle (2pt);
      \filldraw[fill=white, draw=black] (1.2,2.8) circle (2pt);

      \node[red] at (0.2,3.85) {$P_u$};
      \node[blue] at (0.2,2.55) {$P_w$};

      \draw[thick, -{Latex}] (-0.1,2.4) -- (-0.1,1.8);
      \node[black] at (-0.1,-0.1) {$(b)$};

      \draw[red, ultra thick] (-2.4,1.0) -- (-1.2,1.4);
      \draw[blue, ultra thick] (-2.4,1.0) -- (-1.2,0.6);
      \draw[green!70!black, ultra thick] (-2.19,1   ) -- (-1.26,1.3);
      \draw[green!70!black, ultra thick] (-2.2 ,1.01) -- (-1.26,0.7);

      \draw[decorate, decoration={snake, segment length=6pt, amplitude=1.5pt}, red, ultra thick] (-1.2,1.4) -- (1.2,1.4);
      \draw[decorate, decoration={snake, segment length=6pt, amplitude=1.5pt}, blue, ultra thick] (-1.2,0.6) -- (1.2,0.6);

      \filldraw[fill=white, draw=black] (-2.4,1.0) circle (2pt) node[left] {$v$};
      \filldraw[fill=white, draw=black] (-1.2,1.4) circle (2pt) node[above] {$u$};
      \filldraw[fill=white, draw=black] (-1.2,0.6) circle (2pt) node[below] {$w$};
      \filldraw[fill=white, draw=black] (1.2,1.4) circle (2pt);
      \filldraw[fill=white, draw=black] (1.2,0.6) circle (2pt);

      \node[red] at (0.2,1.65) {$P_1$};
      \node[blue] at (0.2,0.35) {$P_2$};
    \end{scope}

    \begin{scope}[xshift=10.2cm,yshift=-10mm]
      \draw[lightgray, ultra thick] (-1.8,4.2) -- (-0.6,4.2); 
      \draw[lightgray, ultra thick] (0.6,4.2) -- (-0.6,4.2);  

      \draw[decorate, decoration={snake, segment length=6pt, amplitude=1.5pt}, red, ultra thick] (-3.2,4.2) -- (-1.8,4.2); 
      \draw[decorate, decoration={snake, segment length=6pt, amplitude=1.5pt}, blue, ultra thick] (0.6,4.2) -- (2.0,4.2);  

      \filldraw[fill=white, draw=black] (-3.2,4.2) circle (2pt);
      \filldraw[fill=white, draw=black] (-1.8,4.2) circle (2pt) node[below] {$u$};
      \filldraw[fill=white, draw=black] (-0.6,4.2) circle (2pt) node[below] {$v$};
      \filldraw[fill=white, draw=black] ( 0.6,4.2) circle (2pt) node[below] {$w$};
      \filldraw[fill=white, draw=black] ( 2.0,4.2) circle (2pt);

      \node[red] at (-2.5,4.45) {$P_u$};
      \node[blue] at (1.3,4.45) {$P_w$};

      \draw[thick, -{Latex}] (-0.6,3.4) -- (-0.6,2.8);
      \node[black] at (-0.6,0.9) {$(c)$};

      \draw[decorate, decoration={snake, segment length=6pt, amplitude=1.5pt}, red, ultra thick] (-3.2,2) -- (-1.8,2); 
      \draw[red, ultra thick] (-1.8,1.98) -- (-0.6,1.98);                       
      \draw[green!70!black, ultra thick] (-1.8,2.05) -- ( 0.6,2.05);            
      \draw[blue, ultra thick] (-0.6,1.98) -- (0.6,1.98);                       
      \draw[decorate, decoration={snake, segment length=6pt, amplitude=1.5pt}, blue, ultra thick] (0.6,2) -- (2.0,2); 

      \filldraw[fill=white, draw=black] (-3.2,2) circle (2pt);
      \filldraw[fill=white, draw=black] (-1.8,2) circle (2pt) node[below] {$u$};
      \filldraw[fill=white, draw=black] (-0.6,2) circle (2pt) node[below] {$v$};
      \filldraw[fill=white, draw=black] ( 0.6,2) circle (2pt) node[below] {$w$};
      \filldraw[fill=white, draw=black] ( 2.0,2) circle (2pt);

      \node[red] at (-2.5,2.25) {$P_u$};
      \node[blue] at (1.3,2.25) {$P_w$};
    \end{scope}
\end{tikzpicture}
\caption{Illustration of the cases in the proof of Theorem~\ref{thm:2degenerate-strong}.}
\label{fig:3cases}
\end{figure}

  Next, suppose that the removal of any vertex $v$ with degree at most~$2$ disconnects~$G$. 
  This means that the minimum degree of~$G$ is~2 and that the removal of a vertex~$v$
  of degree~2 produces two components~$G_1$ and~$G_2$ of size $n_1$ and $n_2$, where each component has
  a single neighbor of~$v$ and $n_1+n_2=n-1$. Note that $n_1, n_2 \geq 3$, as otherwise there would be
  a vertex of degree~1 in~$G$ (either the neighbor of~$v$ or the second vertex in such a small component).
  By induction, let $\calP_1$ and $\calP_2$ be strongly separating path systems
  for $G_1$ and~$G_2$, respectively, with the required properties.
  Let~$u$ and $w$ be the neighbors of $v$ in $G_1$ and~$G_2$, respectively,
  and consider paths $P_u$ in~$\calP_1$ ending at~$u$ and $P_w$ in $\calP_2$
  ending at $w$. Let $P_1$ and~$P_2$ be the paths that extend~$P_u$ and $P_w$
  to~$v$ using the edges $uv$ and $vw$, respectively (see Figure~\ref{fig:3cases}(c)).
  Consider the path family
  $\calP:=(\calP_1\cup \calP_2 \setminus\{P_u,P_w\}) \cup \{P_1,P_2,(u,v,w)\}$.
  It is easy to check that $\calP$ is a strongly separating path system for $G$
  satisfying properties \ref{item:i} and~\ref{item:ii}.
\end{proof}

Let us conclude this section with a simple lemma that will allow us to
find separating path systems with at most $n$ paths for $n$-vertex
connected cubic graphs other than $K_4$ in the next section.

\begin{lemma}\label{obs:2degenerate}
  Let $G$ be a connected cubic $n$-vertex graph that is not $K_4$.  If
  $e=uv$ is an edge of~$G$ that is not in a triangle, then there is a
  strongly separating path system for~$H:=G-e$ with at most~$n$ paths,
  satisfying properties \ref{item:i} and~\ref{item:ii} from
  Theorem~\ref{thm:2degenerate-strong}. Furthermore, such system
  contains two paths of length $2$, one with internal vertex~$u$ and
  the other with internal vertex~$v$.
\end{lemma}
\begin{proof}
  Denote the neighbors of~$u$ and~$v$ respectively by $u_1,u_2$ and
  $v_1,v_2$.  Because $e$ is not in a triangle, $\{u_1,u_2\}$ and
  $\{v_1,v_2\}$ are disjoint. Let $H'$ be the graph obtained by
  removing $u$ and~$v$ from $G$.  Observe that each component of $H'$ is
  2-degenerate, and has at least three vertices because $G$ is not
  $K_4$.  Then, we may apply Theorem~\ref{thm:2degenerate-strong} to
  each component of~$H'$, obtaining path systems with the stated
  properties. Joining these systems gives a strongly separating path
  system~$\calP'$ for~$H'$ with at most $n-2$ paths and the desired
  properties. Such system contains two paths ending at~$u_1$ and two
  paths ending at~$u_2$.  So, we can take two distinct paths $P_{u_1}$
  and~$P_{u_2}$ ending respectively at~$u_1$ and~$u_2$, and also two
  distinct paths $P_{v_1}$ and~$P_{v_2}$ ending respectively at~$v_1$
  and~$v_2$.  We extend $P_{u_1}$ to~$u$, and $P_{v_1}$ to $v$, and we
  extend $P_{u_2}$ to~$u$, and $P_{v_2}$ to $v$. Then, we just need to
  add the paths $(u_1,u,u_2)$ and $(v_1,v,v_2)$ of length $2$ to make
  sure all pairs of edges of $H$ are separated and then we obtain the desired
  strongly separating path system for~$H$ with $n$ paths.
\end{proof}

\section{Corollaries of the result for $2$-degenerate graphs}\label{sec3}

In this section we obtain some results that follow from those of the preceding section.

\begin{theorem}
  \label{thm:sub}
  If $G$ is a connected cubic $n$-vertex graph that is not $K_4$, then~$\ssp(G) \leq n$.
\end{theorem}

\begin{proof}
  As $K_4$ is the only connected cubic graph in which every edge lies
  in a triangle, $G$ contains an edge $e=uv$ that belongs to no
  triangle.  Then, $u$ and $v$ have disjoint neighborhoods
  in~$H := G-e$, which are denoted $\{u_1,u_2\}$ and $\{v_1,v_2\}$,
  respectively.  By Lemma~\ref{obs:2degenerate}, there is a strongly
  separating path system~$\calP'$ for $H$ with $n$ paths, which
  satisfies properties \ref{item:i} and~\ref{item:ii} from
  Theorem~\ref{thm:2degenerate-strong}. Furthermore, the family $\calP'$
  contains two paths, $P_u$ and~$P_v$, of length~$2$ with internal
  vertex $u$ and~$v$, respectively.  Our aim is to show that these
  paths of length~$2$ may be re-routed as paths of length~$3$ using
  $e$ in a way that preserves strong separation (see
  Figure~\ref{fig:second}).  Consider the two paths $P_1$ and $P_2$
  in $\calP'$ starting at $u$, and assume without loss of generality that $P_1$
  uses the edge $uu_1$ and $P_2$ uses the edge $uu_2$.  Let~$Q_1$ and
  $Q_2$ be the corresponding paths for $v$.  If~$P_1=Q_1$, then
  $P_1 \neq Q_2$ and $P_2 \neq Q_1$. In this case, interchange $v_1$
  and $v_2$ (and hence $Q_1$ and $Q_2$) so that $P_1 \neq Q_1$
  and~$P_2 \neq Q_2$.  Consider the path system
  $\mathcal{P}:=\mathcal{P}'\setminus\{P_u,P_v\}\cup\{(u_1,u,v,v_1),(u_2,u,v,v_2)\}$.
  We claim that~$\calP$ is strongly separating for~$G$. It is easy to
  see that $uv$ is strongly separated from any edge of~$H$. Moreover,
  any two edges that are strongly separated in~$H$ are clearly
  strongly separated by the same paths (or by the paths that replaced
  them), except for the pairs $\{uu_1,vv_1\}$ and $\{uu_2,vv_2\}$,
  which are separated by $P_u$ and $P_v$ in~$\mathcal{P}'$, but lie on
  the same new paths. However, our choice of re-routing guarantees
  that the pair $\{uu_1,vv_1\}$ is strongly separated by $P_1$
  and~$Q_1$, while the pair $\{uu_2,vv_2\}$ is strongly separated by
  $P_2$ and $Q_2$.
\end{proof}

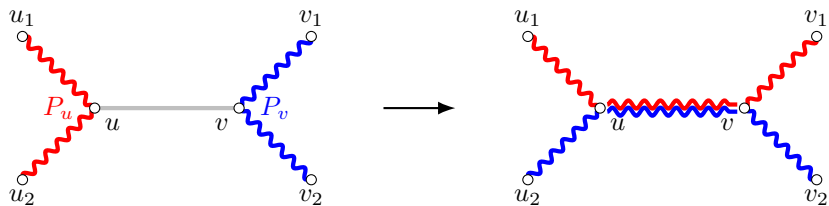
\begin{figure}[htbp]
  \centering
  \begin{tikzpicture}[scale=.95, every node/.style={font=\small}]
    \draw[lightgray, ultra thick] (-1,0) -- (1,0); 

    \draw[decorate, decoration={snake, segment length=6pt, amplitude=1.5pt}, red, ultra thick] (-2, 1) -- (-1,0);
    \draw[decorate, decoration={snake, segment length=6pt, amplitude=1.5pt}, red, ultra thick] (-2,-1) -- (-1,0);
    \node[red] at (-1.5,0) {$P_u$};

    \draw[decorate, decoration={snake, segment length=6pt, amplitude=1.5pt}, blue, ultra thick] (2, 1) -- (1,0);
    \draw[decorate, decoration={snake, segment length=6pt, amplitude=1.5pt}, blue, ultra thick] (2,-1) -- (1,0);
    \node[blue] at (1.5,0) {$P_v$};

    \filldraw[fill=white, draw=black] (-1,0) circle (2pt) node[anchor=north west] {$u$};
    \filldraw[fill=white, draw=black] ( 1,0) circle (2pt) node[anchor=north east] {$v$};
    \filldraw[fill=white, draw=black] (-2, 1) circle (2pt) node[above] {$u_1$};
    \filldraw[fill=white, draw=black] (-2,-1) circle (2pt) node[below] {$u_2$};
    \filldraw[fill=white, draw=black] (2, 1) circle (2pt) node[above] {$v_1$};
    \filldraw[fill=white, draw=black] (2,-1) circle (2pt) node[below] {$v_2$};

    \draw[thick, -{Latex}] (3,0) -- (4,0);

  \begin{scope}[xshift=7cm]
    \draw[decorate, decoration={snake, segment length=6pt, amplitude=1.5pt}, red, ultra thick]  (-0.9, 0.05) -- (0.9, 0.05);
    \draw[decorate, decoration={snake, segment length=6pt, amplitude=1.5pt}, blue, ultra thick] (-0.9,-0.05) -- (0.9,-0.05);
    \draw[decorate, decoration={snake, segment length=6pt, amplitude=1.5pt}, red, ultra thick] (-2, 1) -- (-1,0);
    \draw[decorate, decoration={snake, segment length=6pt, amplitude=1.5pt}, red, ultra thick] (2, 1) -- (1,0);
    \draw[decorate, decoration={snake, segment length=6pt, amplitude=1.5pt}, blue, ultra thick] (-2,-1) -- (-1,0);
    \draw[decorate, decoration={snake, segment length=6pt, amplitude=1.5pt}, blue, ultra thick] (2,-1) -- (1,0);
    \filldraw[fill=white, draw=black] (-1,0) circle (2pt) node[anchor=north west] {$u$};
    \filldraw[fill=white, draw=black] ( 1,0) circle (2pt) node[anchor=north east] {$v$};
    \filldraw[fill=white, draw=black] (-2, 1) circle (2pt) node[above] {$u_1$};
    \filldraw[fill=white, draw=black] (-2,-1) circle (2pt) node[below] {$u_2$};
    \filldraw[fill=white, draw=black] (2, 1) circle (2pt) node[above] {$v_1$};
    \filldraw[fill=white, draw=black] (2,-1) circle (2pt) node[below] {$v_2$};
  \end{scope}
\end{tikzpicture}
\caption{Illustration for the proof of Theorem~\ref{thm:sub}.}
\label{fig:second}
\end{figure}

We say that a graph is \emph{subcubic}
if all of its vertices have degree at most $3$. 
We obtain the following corollary from Theorem~\ref{thm:sub}.

\begin{corollary}\label{cor:subcubic}
  If $G$ is a subcubic $n$-vertex graph with $k$ connected components
  isomorphic to~$K_4$, then $\ssp(G) \leq n + k$.
\end{corollary}
\begin{proof}
  Since $K_4$ requires $5$ paths to be strongly separated,
  we need $5k$ paths for the components isomorphic to~$K_4$.
  Furthermore, in view of Theorem~\ref{thm:sub}, we can strongly
  separate the remaining edges with $n - 4k$ paths,
  for a total of $5k + (n - 4k) = n + k$ paths.
\end{proof}

Outerplanar graphs always contain a vertex of degree at most~2,
and their class is closed under subgraphs, so they are 2-degenerate.
Therefore, the following is a direct corollary of Theorem~\ref{thm:2degenerate}.

\begin{corollary}\label{cor:outerplanar}
If $G$ is an outerplanar $n$-vertex graph, then $\ssp(G) \leq n$.
\end{corollary}

In 2005, Gon\c calves \cite{gonccalves2005edge} proved that every planar graph can be partitioned into two outerplanar graphs, resolving a conjecture of Chartrand, Geller, and Hedetniemi \cite{chartrand1971graphs}. As a result, we have:

\begin{corollary}
If $G$ is a planar $n$-vertex graph, then $\ssp(G) \leq 2n$.
\end{corollary}

Following Krisam~\cite{Krisam2021},
in every bipartite planar $n$-vertex graph, there are at most $n/2$ edges
whose removal results in a 2-degenerate graph.
This and Theorem~\ref{thm:2degenerate} imply the following.

\begin{corollary}
If $G$ is a bipartite planar $n$-vertex graph, then $\ssp(G) \leq 3n/2$.
\end{corollary}

In fact, Krisam~\cite{Krisam2021} conjectured that, in every bipartite planar
$n$-vertex graph~$G$, the removal of~$n/4$ edges would be enough
to obtain a 2-degenerate graph, which would imply that ${\ssp(G) \leq 5n/4}$.


\section{Bounds for bipartite graphs} \label{sec4}

As mentioned in the introduction, it was proved in \cite{FSMM25} that
$\ssp(K_{\frac{n}{2}, \frac{n}{2}})=(\sqrt{5/2} - 1 + o(1))n \approx 0.58\,n$.
This result works only for large $n$ and it does not apply to complete
bipartite graphs that are not almost regular.
We are interested in obtaining good bounds for all complete bipartite graphs. 

The main result of this section is twofold: for
$1 \leq a \leq b/2$, we show that $\ssp(K_{a,b}) \leq b$,
which in view of the fact that $\ssp(K_{a,b}) \geq \Delta(K_{a,b})$,
determines $\ssp(K_{a,b})$ in this range of~$a$.  On the other hand,
when $a > b/2$, it is no longer true that $\ssp(K_{a,b}) = b$,
and we provide a sequence of lower bounds on $\ssp(K_{a,b})$, for any
possible $a$ in this range, which matches exactly the values of
$\ssp(K_{a,b})$ for $a=b/2$ and $a=b$.

\begin{theorem}
  \label{thm:bip}
  The following hold for any positive integers $a$ and $b$ with $1 \leq a \leq b$.
  \begin{enumerate}
  \item If $a < b/2$, then $\ssp(K_{a,b}) = b$;\label{itemi}
  \item If $a \geq b/2$, then $\ssp(K_{a,b}) \geq (\sqrt{6(b/a)+4}-2)a$.\label{itemii}
  \end{enumerate}
\end{theorem}

In Figure~\ref{fig:lower}, we show the lower bounds given in
Theorem~\ref{thm:bip}.  For its proof, we will use a \emph{graceful
  labeling}, which, given a graph $G$ with $\ell$ edges, is an injective
assignment~$\phi$ of labels from $\{0,1,\dots,\ell\}$ to $V(G)$ such
that each edge $uv$ of $G$ induces a distinct value
$|\phi(u)-\phi(v)|$, which lies in~$[\ell]$, where
$[\ell]=\{1,\ldots,\ell\}$.

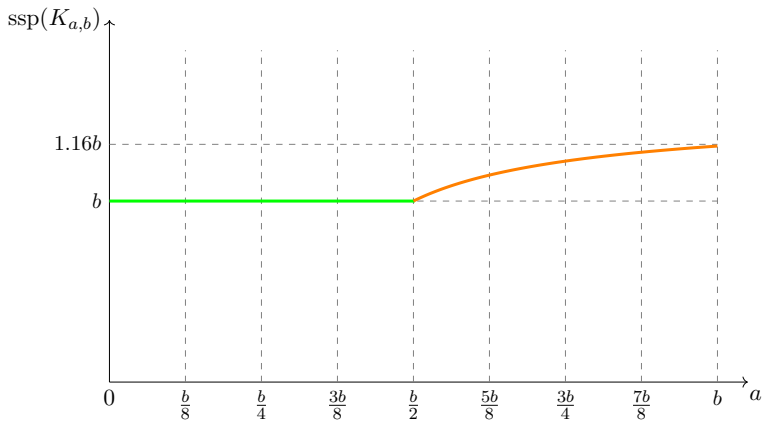
\begin{figure}[h!]
\centering
  \scalebox{.8}{
    \begin{tikzpicture}
  \draw[->] (0, 0) -- (10.5, 0) node[anchor=north] {\ \ $a$};
  \draw[->] (0, 0) -- (0, 6) node[anchor=east] {$\ssp(K_{a, b})$};

  \foreach \x in {1.25, 2.5, 3.75, 5, 6.25, 7.5, 8.75, 10} {
    \draw[dashed, gray] (\x, 0) -- (\x, 5.5);
  }
  \draw[dashed, gray] (0, 3) -- (10, 3); 
  \draw[dashed, gray] (0,  3.94) -- (10,  3.94); 

  \node[anchor=north] at (0, 0) {$0$};
  \node[anchor=north] at (1.25, 0) {$\frac{b}{8}$};
  \node[anchor=north] at (2.5, 0) {$\frac{b}{4}$};
  \node[anchor=north] at (3.75, 0) {$\frac{3b}{8}$};
  \node[anchor=north] at (5, 0) {$\frac{b}{2}$};
  \node[anchor=north] at (6.25, 0) {$\frac{5b}{8}$};
  \node[anchor=north] at (7.5, 0) {$\frac{3b}{4}$};
  \node[anchor=north] at (8.75, 0) {$\frac{7b}{8}$};
  \node[anchor=north] at (10, 0) {\small{$b$}};

  \node[anchor=east] at (0, 3) {\small{$b$}};
  \node[anchor=east] at (0, 3.97367) {\small{$1.16b$}};

  \draw[line width=0.5mm, green] (0, 3) -- (5.02, 3);

  \draw[line width=0.5mm, orange, domain=5:10, samples=100, smooth, variable=\x]
    plot ({\x}, {((\x-3.5) * (sqrt(6*(3/(\x-3.5))+4)-2))});
  \end{tikzpicture}
  }
  \caption{Lower bounds for $\ssp(K_{a,b})$ for all possible values
    of $a$ and $b$.}
  \label{fig:lower}
\end{figure}

\begin{proof}[Proof of Theorem~\ref{thm:bip}.]
  By~\cite[Theorem 5]{BCMP2016}, we know that $\ssp(K_{1,b})=b$, so that case (i) of Theorem~\ref{thm:bip} immediately holds for $a = 1$ and every $b \geq 3$. One may easily verify that case (ii) holds for $a=1$ and $b=2$. Moreover, for $a=b=2$, it is easy to see that $\ssp(K_{2,2})=4>2\sqrt{10}-4$, satisfying case (ii).
  
  Given $2 \leq a\leq b$ with $b \geq 3$, from now on we consider the complete
  bipartite graph $K_{a,b}$ with bipartition $\{u_0,u_1,\dots,u_{a-1}\}$,
  $\{v_0,v_1,\dots,v_{b-1}\}$.

  First, let $a < b/2$ and, for any $k\in\mathbb{N}$, define
  the path with $k$ edges $P(k):=(0,1,\dots,k)$. We pick an
  arbitrary graceful labeling $\phi$ of the path $P(a)$. 
  We show graceful labelings in Figure~\ref{fig:grace} for
  paths with even and odd number of edges.

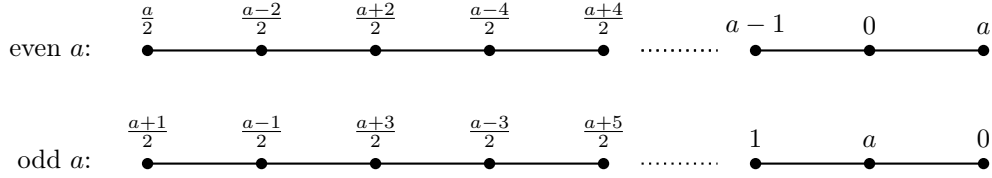
\begin{figure}[ht]
\centering
 \begin{tikzpicture}
    \tikzstyle{dot} = [circle, fill, black, inner sep=1.5pt]

    \draw[thick] (-11,0) -- (-5,0);
    \draw[thick] (-3,0) -- (0,0);
    \draw[thick, dotted] (-4.5,0) -- (-3.5,0);

    \node[anchor=east] at (-11.6,0) {\small even $a$:};
    \node[anchor=east] at (-11.6,-1.44) {\small odd $a$:};
    \node[dot, label=above:{\small$\frac{a}{2}$}] (m2) at (-11,0) {};
    \node[dot, label=above:{\small$\frac{a-2}{2}$}] (m2_minus_1) at (-9.5,0) {};
    \node[dot, label=above:{\small$\frac{a+2}{2}$}] (m2_plus_1) at (-8,0) {};
    \node[dot, label=above:{\small$\frac{a-4}{2}$}] (m2_minus_2) at (-6.5,0) {};
    \node[dot, label=above:{\small$\frac{a+4}{2}$}] (m2_plus_2) at (-5,0) {};

    \node[dot, label=above:{\small$a-1$}] (m_minus_1) at (-3,0) {};
    \node[dot, label=above:{\small$0$}] (zero) at (-1.5,0) {};
    \node[dot, label=above:{\small$a$}] (m) at (0,0) {};

    \draw[thick, dotted] (-4.5,0) -- (-3.5,0);


    \draw[thick] (-11,-1.5) -- (-5,-1.5);
    \draw[thick] (-3,-1.5) -- (0,-1.5);
    \draw[thick, dotted] (-4.5,-1.5) -- (-3.5,-1.5);

    \node[dot, label=above:{\small$\frac{a+1}{2}$}] (m2) at (-11,-1.5) {};
    \node[dot, label=above:{\small$\frac{a-1}{2}$}] (m2_minus_1) at (-9.5,-1.5) {};
    \node[dot, label=above:{\small$\frac{a+3}{2}$}] (m2_plus_1) at (-8,-1.5) {};
    \node[dot, label=above:{\small$\frac{a-3}{2}$}] (m2_minus_2) at (-6.5,-1.5) {};
    \node[dot, label=above:{\small$\frac{a+5}{2}$}] (m2_plus_2) at (-5,-1.5) {};

    \node[dot, label=above:{\small$1$}] (m_minus_1) at (-3,-1.5) {};
    \node[dot, label=above:{\small$a$}] (zero) at (-1.5,-1.5) {};
    \node[dot, label=above:{\small$0$}] (m) at (0,-1.5) {};
  \end{tikzpicture}
  \caption{Graceful labelings of the path $P(a)$.}
  \label{fig:grace}
\end{figure}

For $0 \leq j < b$, define
$P_j:=(v_{\phi(0)+j},u_0,v_{\phi(1)+j},u_1,\ldots,v_{\phi(a-1)+j},u_{a-1},v_{\phi(a)+j})$, 
where the indices of the $v$-vertices are taken modulo $b$. 
Furthermore, consider the family of paths 
$\mathcal{P}:=\{P_j : 0\leq j < b\}$ in~$K_{a,b}$.  
Look at Figure~\ref{fig:example} for such a collection for $K_{3,7}$, 
based on the graceful labeling described in Figure~\ref{fig:grace}.
We claim that $\mathcal{P}$ is a strongly separating path system for~$K_{a,b}$
and, since $|\mathcal{P}| = b$, this suffices to complete the proof of~\ref{itemi}.

\begin{figure}[ht]
\centering
 \begin{tikzpicture}
    \tikzstyle{dot} = [circle, fill, black, inner sep=1.5pt]

    \node[anchor=west] at (-4.3,6.5) {\small graceful labeling $\phi$:};
    \draw[thick] (-5,6) -- (-0.5,6);
    \node[dot, label=below:{\small$2$}] at (-5,6) {};
    \node[dot, label=below:{\small$1$}] at (-3.5,6) {};
    \node[dot, label=below:{\small$3$}] at (-2,6) {};
    \node[dot, label=below:{\small$0$}] at (-0.5,6) {};

    \node[dot, label=above:{\small$u_0$}] (u0) at (3.2,6.5) {};
    \node[dot, label=above:{\small$u_1$}] (u1) at (3.8,6.5) {};
    \node[dot, label=above:{\small$u_2$}] (u2) at (4.4,6.5) {};

    \node[dot, label=below:{\small$v_0$}] (v0) at (2,5.5) {};
    \node[dot, label=below:{\small$v_1$}] (v1) at (2.6,5.5) {};
    \node[dot, label=below:{\small$v_2$}] (v2) at (3.2,5.5) {};
    \node[dot, label=below:{\small$v_3$}] (v3) at (3.8,5.5) {};
    \node[dot, label=below:{\small$v_4$}] (v4) at (4.4,5.5) {};
    \node[dot, label=below:{\small$v_5$}] (v5) at (5,5.5) {};
    \node[dot, label=below:{\small$v_6$}] (v6) at (5.6,5.5) {};

    \draw[thick] (v2) -- (u0) -- (v1) -- (u1) -- (v3) -- (u2) -- (v0);
    \node[label={$P_0$}] at (4.7,5.6) {};

    \node[dot, label=above:{\small$u_0$}] (u0) at (-5.2,4) {};
    \node[dot, label=above:{\small$u_1$}] (u1) at (-4.6,4) {};
    \node[dot, label=above:{\small$u_2$}] (u2) at (-4,4) {};

    \node[dot, label=below:{\small$v_0$}] (v0) at (-6.4,3) {};
    \node[dot, label=below:{\small$v_1$}] (v1) at (-5.8,3) {};
    \node[dot, label=below:{\small$v_2$}] (v2) at (-5.2,3) {};
    \node[dot, label=below:{\small$v_3$}] (v3) at (-4.6,3) {};
    \node[dot, label=below:{\small$v_4$}] (v4) at (-4,3) {};
    \node[dot, label=below:{\small$v_5$}] (v5) at (-3.4,3) {};
    \node[dot, label=below:{\small$v_6$}] (v6) at (-2.8,3) {};

    \draw[thick] (v3) -- (u0) -- (v2) -- (u1) -- (v4) -- (u2) -- (v1);
    \node[label={$P_1$}] at (-3.6,3.1) {};

    \node[dot, label=above:{\small$u_0$}] (u0) at (-0.2,4) {};
    \node[dot, label=above:{\small$u_1$}] (u1) at (0.4,4) {};
    \node[dot, label=above:{\small$u_2$}] (u2) at (1,4) {};

    \node[dot, label=below:{\small$v_0$}] (v0) at (-1.4,3) {};
    \node[dot, label=below:{\small$v_1$}] (v1) at (-0.8,3) {};
    \node[dot, label=below:{\small$v_2$}] (v2) at (-0.2,3) {};
    \node[dot, label=below:{\small$v_3$}] (v3) at (0.4,3) {};
    \node[dot, label=below:{\small$v_4$}] (v4) at (1,3) {};
    \node[dot, label=below:{\small$v_5$}] (v5) at (1.6,3) {};
    \node[dot, label=below:{\small$v_6$}] (v6) at (2.2,3) {};

    \draw[thick] (v4) -- (u0) -- (v3) -- (u1) -- (v5) -- (u2) -- (v2);
    \node[label={$P_2$}] at (1.7,3.2) {};

    \node[dot, label=above:{\small$u_0$}] (u0) at (4.8,4) {};
    \node[dot, label=above:{\small$u_1$}] (u1) at (5.4,4) {};
    \node[dot, label=above:{\small$u_2$}] (u2) at (6,4) {};

    \node[dot, label=below:{\small$v_0$}] (v0) at (3.6,3) {};
    \node[dot, label=below:{\small$v_1$}] (v1) at (4.2,3) {};
    \node[dot, label=below:{\small$v_2$}] (v2) at (4.8,3) {};
    \node[dot, label=below:{\small$v_3$}] (v3) at (5.4,3) {};
    \node[dot, label=below:{\small$v_4$}] (v4) at (6,3) {};
    \node[dot, label=below:{\small$v_5$}] (v5) at (6.6,3) {};
    \node[dot, label=below:{\small$v_6$}] (v6) at (7.2,3) {};

    \draw[thick] (v5) -- (u0) -- (v4) -- (u1) -- (v6) -- (u2) -- (v3);
    \node[label={$P_3$}] at (7,3.2) {};

    \node[dot, label=above:{\small$u_0$}] (u0) at (-5.2,1.5) {};
    \node[dot, label=above:{\small$u_1$}] (u1) at (-4.6,1.5) {};
    \node[dot, label=above:{\small$u_2$}] (u2) at (-4,1.5) {};

    \node[dot, label=below:{\small$v_0$}] (v0) at (-6.4,0.5) {};
    \node[dot, label=below:{\small$v_1$}] (v1) at (-5.8,0.5) {};
    \node[dot, label=below:{\small$v_2$}] (v2) at (-5.2,0.5) {};
    \node[dot, label=below:{\small$v_3$}] (v3) at (-4.6,0.5) {};
    \node[dot, label=below:{\small$v_4$}] (v4) at (-4,0.5) {};
    \node[dot, label=below:{\small$v_5$}] (v5) at (-3.4,0.5) {};
    \node[dot, label=below:{\small$v_6$}] (v6) at (-2.8,0.5) {};

    \draw[thick] (v6) -- (u0) -- (v5) -- (u1) -- (v0) -- (u2) -- (v4);
    \node[label={$P_4$}] at (-3.6,0.7) {};

    \node[dot, label=above:{\small$u_0$}] (u0) at (-0.2,1.5) {};
    \node[dot, label=above:{\small$u_1$}] (u1) at (0.4,1.5) {};
    \node[dot, label=above:{\small$u_2$}] (u2) at (1,1.5) {};

    \node[dot, label=below:{\small$v_0$}] (v0) at (-1.4,0.5) {};
    \node[dot, label=below:{\small$v_1$}] (v1) at (-0.8,0.5) {};
    \node[dot, label=below:{\small$v_2$}] (v2) at (-0.2,0.5) {};
    \node[dot, label=below:{\small$v_3$}] (v3) at (0.4,0.5) {};
    \node[dot, label=below:{\small$v_4$}] (v4) at (1,0.5) {};
    \node[dot, label=below:{\small$v_5$}] (v5) at (1.6,0.5) {};
    \node[dot, label=below:{\small$v_6$}] (v6) at (2.2,0.5) {};

    \draw[thick] (v0) -- (u0) -- (v6) -- (u1) -- (v1) -- (u2) -- (v5);
    \node[label={$P_5$}] at (1.7,0.7) {};

    \node[dot, label=above:{\small$u_0$}] (u0) at (4.8,1.5) {};
    \node[dot, label=above:{\small$u_1$}] (u1) at (5.4,1.5) {};
    \node[dot, label=above:{\small$u_2$}] (u2) at (6,1.5) {};

    \node[dot, label=below:{\small$v_0$}] (v0) at (3.6,0.5) {};
    \node[dot, label=below:{\small$v_1$}] (v1) at (4.2,0.5) {};
    \node[dot, label=below:{\small$v_2$}] (v2) at (4.8,0.5) {};
    \node[dot, label=below:{\small$v_3$}] (v3) at (5.4,0.5) {};
    \node[dot, label=below:{\small$v_4$}] (v4) at (6,0.5) {};
    \node[dot, label=below:{\small$v_5$}] (v5) at (6.6,0.5) {};
    \node[dot, label=below:{\small$v_6$}] (v6) at (7.2,0.5) {};

    \draw[thick] (v1) -- (u0) -- (v0) -- (u1) -- (v2) -- (u2) -- (v6);
    \node[label={$P_6$}] at (7,0.7) {};
  \end{tikzpicture}
  \caption{Collection $\calP$ of paths for $K_{3,7}$.}
  \label{fig:example}
\end{figure}

Since $\phi$ is a graceful labeling, hence an injection, each element
of $\mathcal{P}$ is a copy of $P(2a)$. But, since $a<b/2$, each
edge $u_iv_j$ of $K_{a,b}$ appears in exactly two paths of
$\mathcal{P}$, namely in~$P_{j-\phi(i)}$ and in~$P_{j-\phi(i+1)}$,
where the indices are taken modulo $b$.

Let $u_iv_j$ and $u_{i'}v_{j'}$ be two edges of $K_{a,b}$ and suppose
towards a contradiction that they appear in the same two paths
of~$\calP$. If $i=i'$, then $j\neq j'$, so it must hold that
$j-\phi(i)=j'-\phi(i+1)$ and $j'-\phi(i)=j-\phi(i+1)$, both modulo
$b$, hence $2(\phi(i+1)-\phi(i))=0$ modulo $b$. But, since
$a<b/2$ and $|\phi(i+1)-\phi(i)| \in [a]$, the fact that
$\phi$ is a graceful labeling of $P(a)$ leads to a contradiction. On
the other hand, if $i\neq i'$, then, similarly, we have
$\phi(i+1)-\phi(i)=\phi(i'+1)-\phi(i')$ or
$\phi(i+1)-\phi(i)=\phi(i'+1)-\phi(i')$, both modulo~$b$.  Again,
since $a < b/2$, the above quantity resides in $[2a]$, yielding
a contradiction as $\phi$ is a graceful labeling.  We conclude that no
two edges share both of their paths in $\calP$ and therefore $\calP$
is a strongly separating path system for $K_{a,b}$.

Now let $a \geq b/2$ and let $\mathcal{P}$ be an arbitrary
strongly separating path system for $K_{a,b}$.  Let
$p:=|\mathcal{P}|$. Denote by $e_i$ the number of edges that appear in
exactly $i$ paths of $\mathcal{P}$. Then, we have
\begin{equation}\label{eq:1}
 3ab-2e_1-e_2 \ = \ e_1 + 2e_2 + 3(ab-e_1-e_2) \ \leq \ \sum_{i=1}^p i\cdot e_i\ \leq \ 2ap,
\end{equation}
where the left and the right side of this inequality are,
respectively, lower and upper bounds for the number of edges (with
multiplicity) in $\calP$. Moreover, every edge that appears in exactly
one path of $\calP$ must be itself a path in $\calP$, and each two
edges that appear in exactly two paths of $\calP$ must belong to a
different pair of paths, so
\begin{equation}\label{eq:21}
    e_2+2e_1 \ \leq \ {p-e_1\choose 2}+2e_1 \ = \ 
                   \frac{p^2}{2}+\frac{e_1}{2}(e_1-2p)+\frac{5e_1-p}{2}.
\end{equation}
Given that $e_1 \leq p$, the right side of~\eqref{eq:21} is bounded above by
$$\frac{p^2}{2}-\frac{e_1^2}{2}+2e_1,$$
which is at most $p^2/2$ for every $e_1 \geq 4$.

Next assume that $e_1 \leq 3$. Recall that $a \geq 2$ and $b \geq 3$, so that $e(K_{a,b})\geq 6$. There must be at least three other paths in the family to separate the edges that do not belong to a path of length 1, so that $p \geq e_1+3$. The upper bound in~\eqref{eq:21} is bounded above by
$$\frac{p^2}{2}-\left(\frac{e_1^2}{2}+3e_1\right)+\left(2e_1-\frac{3}{2}\right) \ < \ \frac{p^2}{2}.$$

In all cases, we get 
\begin{equation}\label{eq:2}
    e_2+2e_1\ \leq \frac{p^2}{2}.
\end{equation}
Combining \eqref{eq:1} and \eqref{eq:2}, we get $p^2+4ap-6ab \geq 0$,
which gives $p\geq (\sqrt{6(b/a)+4}-2)a$, concluding the proof of~\ref{itemii}.
\end{proof}

The lower bounds given in Theorem~\ref{thm:bip} are tight if $a = b/2$
or $a = b$.  It would be interesting to find tight examples for every $a$
with $b/2 < a < b$, or
to improve the lower bound in these cases.

\section*{Acknowledgements} 

We thank the reviewers for their valuable insights in the paper.

\bibliographystyle{plain}
\bibliography{main}

\end{document}